\documentclass[12pt]{amsart} 
\usepackage{latexsym,mathtools}
\usepackage{epsfig,setspace}
\usepackage[utf8]{inputenc}
\usepackage{tipa}
\usepackage{a4}
\usepackage{tikz-cd}
\usepackage{amssymb}
\usepackage{amsthm}
\usepackage{amsbsy}
\usepackage{upgreek}
\usepackage{amsfonts,amssymb,latexsym,amscd,amsmath,euscript,enumerate,verbatim,calc}
\allowdisplaybreaks[1]
\usepackage{url}
\usepackage{hhline}
\usepackage{pifont}
\usepackage[colorlinks=true,linkcolor=blue,citecolor=magenta]{hyperref}  
\theoremstyle{plain}

\newtheorem{theorem}{Theorem}[section] 
\newtheorem{lemma}[theorem]{Lemma}
\newtheorem{proposition}[theorem]{Proposition}

\newtheorem{corollary}[theorem]{Corollary}
\theoremstyle{definition}
\newtheorem{definition}[theorem]{Definition}

\newtheorem{remark}[theorem]{Remark}

\def\Fq{{\mathbb F}_q}

\def\vx{{\bf x}}
\def\vu{{\bf u}}
\def\Bk{\mathcal{B}_1}
\def\Bn{\mathcal{B}_2}
\def\nqdd{\frac{|\GL_d(\Fq)|}{\prod_{f\in S(\I)}c_f(\Phi_\I(f))}}

\def\I{\mathcal{I}}

\def\IT{\mathfrak{I}(T)}

\newcommand{\diag}{\operatorname{diag}}

\newcommand{\rank}{\operatorname{rank}}
\newcommand{\GL}{\operatorname{GL}}

\makeatletter
\def\imod#1{\allowbreak\mkern10mu({\operator@font mod}\,\,#1)}
\makeatother
 
\title[Linear transformations with prescribed invariant factors]{The number of linear transformations defined on a subspace with given invariant factors}
\author{Samrith Ram} 
\address{Indraprastha Institute of Information Technology Delhi (IIIT-Delhi), New Delhi, India.}
\email{samrith@gmail.com}
\date{\today}



\begin{document} 

\begin{abstract}
Given a finite-dimensional vector space $V$ over the finite field $\mathbb{F}_q$ and a subspace $W$ of $V$, we consider the problem of counting linear transformations $T:W\to V$ which have prescribed invariant factors. The case $W=V$ is a well-studied problem that is essentially equivalent to counting the number of square matrices over $\mathbb{F}_q$ in a conjugacy class and an explicit formula is known in this case. On the other hand, the case of general $W$ is also an interesting problem and there hasn't been substantive progress in this case for over two decades, barring a special case where all the invariant factors of $T$ are of degree zero. We extend this result to the case of arbitrary $W$ by giving an explicit counting formula. As an application of our results, we give new proofs of some recent enumerative results in linear control theory and derive an extension of the Gerstenhaber-Reiner formula for the number of square matrices over $\mathbb{F}_q$ with given characteristic polynomial. 
\end{abstract} 
\maketitle
\section{Introduction}
Let $\Fq$ denote the finite field with $q$ elements and let $\Fq[x]$ denote the ring of polynomials in the indeterminate $x$ with coefficients in $\Fq$. Throughout this paper, $n,k$ denote positive integers with $n\geq k$. For any ring $R$, we denote by $M_{n,k}(R)$ the set of all $n \times k$ matrices with entries in $R$ and by $M_n(R)$ the set of all square $n\times n$ matrices with entries in $R$. Let $\GL_n(\Fq)$ denote the general linear group of nonsingular matrices of $M_n(\Fq)$. 

 Given any nonzero polynomial matrix $A\in M_{n,k}(\Fq[x])$, there exist invertible matrices $P\in M_n(\Fq[x])$ and $Q\in M_k(\Fq[x])$ such that the product $PAQ$ is of the form 
$$
\begin{bmatrix}
  p_1 & 0 & 0& \quad & \cdots & \quad &0\\
  0   & p_2&0& \quad & \cdots & \quad &0\\
  0   & 0  &\ddots& \quad &\quad &\quad & 0\\ 
 \vdots & \quad & \quad &p_t & \quad & \quad & \vdots\\
 \quad & \quad & \quad &\quad & 0 & \quad & \quad\\
 \quad & \quad & \quad &\quad & \quad & \ddots & \quad\\
0 & \quad & \quad &\cdots & \quad & \quad & 0\\
\end{bmatrix}
$$
 where all off-diagonal entries are zero and the nonzero entries on the diagonal are monic polynomials $p_i (1 \leq i \leq t)$ satisfying $p_i\mid p_{i+1}$ for $1\leq i <t$. This diagonal form is known as the Smith normal form of $A$. We express this concisely by writing $A\sim \diag_{n,k}(p_1,\ldots,p_t)$. The $p_i$ are known as the invariant factors of $A$ and are given by
$$
p_i=\frac{\delta_i(A)}{\delta_{i-1}(A)},
$$
where $\delta_i(A)$ denotes the $i$\textsuperscript{th} determinantal divisor of $A$ and equals the greatest common divisor of all $i\times i$ minors of the matrix $A$. 

Let us denote by $I_{n,k}$ the element of $M_{n,k}(\Fq)$ whose $(i,j)$-th entry is 1 for $i=j$ and zero otherwise. In this paper, we are mainly interested in counting the number of matrices $B\in M_{n,k}(\Fq)$ for which the matrix polynomial $xI_{n,k}-B$ has a prescribed list of invariant factors. 
More precisely, for any $k$-tuple of invariant factors, i.e., a tuple $\I=(p_1,\ldots,p_k)$ of monic polynomials in $\Fq[x]$ such that $p_i\mid p_{i+1}(1\leq i \leq k-1)$, define 
 $$
N_q(n,k;\I):=\#\left\{B\in M_{n,k}(\Fq): xI_{n,k}-B \sim \diag_{n,k}(p_1,\ldots,p_k)\right\}.
$$
  We are interested in a formula for $N_q(n,k;\I)$. In fact, the problem of determining $N_q(n,n;\I)$ is equivalent to counting the number of square $n\times n$ matrices over $\Fq$ in a conjugacy class. This problem has been studied by Kung \cite{Kung1981} and Stong \cite{Stong1988} 
among others, and an explicit formula (see \eqref{eq:nqnn}) due to Philip Hall can be found in Stanley \cite{Stanley2012}. 

On the other hand, rectangular matrices arise in many contexts in linear control theory and matrix completion problems and are also of considerable interest. We refer to Cravo \cite[Thms. 15, 32]{Cravo2009} for specific examples of matrix completion problems.  For another example, suppose $k<n$ and $\I_0$ denotes the $k$-tuple $(1,\ldots,1)$. Determining $N_q(n,k;\I_0)$ is equivalent to the problem of counting the number of zero kernel pairs over a finite field. We refer to Gohberg et al. \cite[Sec.~X.1]{Gohbergetal1995} or \cite{zerokernel} for the definition of zero kernel pairs. 
The problem of counting zero kernel pairs goes back at least to the work of Koci\textpolhook{e}cki and Przyłuski \cite{KocPrz1989} who considered the equivalent problem of determining the number of reachable pairs (see Definition \ref{reachable}) of matrices over a finite field. This problem has recently been settled by Helmke et al. \cite[Thm. 1]{Helmkeetal2015} who also solve the more general problem of determining the number of matrix pairs with an $r$-dimensional reachability subspace (see Definition \ref{matrixpair}). 

 The main result of this paper is Theorem \ref{th:main} which is a counting formula for the number of linear transformations defined on a subspace which have a given invariant subspace and prescribed invariant factors. We use this result to deduce Theorem~\ref{snfcount} which gives a formula for $N_q(n,k;\I)$ for an arbitrary $k$-tuple $\I$ of invariant factors. 
We give further applications of our results in Sections \ref{pairs} and \ref{sec:ghformula}. In Section~\ref{pairs} we give an alternate proof of the formula of Helmke et al. (Theorem~\ref{Helmke}) for the number of matrix pairs with an $r$-dimensional reachability subspace. Studying linear maps defined on a subspace allows us to extend some combinatorial results on square matrices to the setting of rectangular matrices. As an illustration, we prove an extension of the Gerstenhaber-Reiner formula for the number of square matrices with given characteristic polynomial in Section~\ref{sec:ghformula}.

\section{Preliminaries and Notation} 
It is a well-known fact that two matrices $A,A'\in M_n(\Fq)$ are similar, i.e., there exists $P\in \GL_n(\Fq)$ with $A'=PAP^{-1}$, if and only if they have the same invariant factors. More precisely, if 
$$
xI_n-A\sim \diag_{n,n}(p_1,\ldots,p_n) \mbox{ and } xI_n-A'\sim \diag_{n,n}(p_1',\ldots,p_n'),
$$
then $A$ is similar to $A'$ if and only if $p_i=p_i'$ for $1\leq i \leq n$. Therefore, the conjugacy class of all matrices similar to $A$ is uniquely determined by the invariant factors of $A$. To give an expression for the number of matrices in the conjugacy class determined by an $n$-tuple of invariant factors, we introduce some notation.

Following Stanley \cite[Sec. 1.10]{Stanley2012}, we denote by $\mathrm{Par}$ the set of all partitions of all nonnegative integers. We write $\lambda \vdash n$ if $\lambda$ is a partition of $n$.
 Given any partition $\lambda\in \mathrm{Par}$, let $\lambda'=(\lambda_1',\lambda_2',\ldots)$ denote the conjugate partition to $\lambda$ and let $m_i=m_i(\lambda)=\lambda_i'-\lambda_{i+1}'$ denote the number of parts of $\lambda$ of size $i$. Let
$$
h_i=h_i(\lambda)=\lambda_1'+\lambda_2'+\cdots+\lambda_i'.
$$
For $\lambda\in \mathrm{Par}$ and any irreducible $f\in \Fq[x]$ of degree $d$, define
$$
c_f(\lambda):=\prod_{i\geq 1}\prod_{j=1}^{m_i}\left(q^{dh_i}-q^{d(h_i-j)}\right).
$$
For $f,g\in \Fq[x]$ with $f$ irreducible and $g$ nonzero, set
$$
\nu_f(g):=\max \{e:f^e\mid g\}.
$$
Given a tuple $\I=(p_1,\ldots,p_k)$ of invariant factors, let $S(\I)$ denote the set of all monic irreducible polynomials which divide $P=\prod_{i=1}^{k}p_i$.
To each $f\in S(\I)$, we associate a partition $\Phi_\I(f)\vdash \nu_f(P)$ by defining
$$
\Phi_\I(f):=\left(\nu_f(p_k),\nu_f(p_{k-1}),\ldots,\nu_f(p_r)\right),
$$
where $r=\min \{i:\nu_f(p_i)>0\}$. The following formula (essentially equivalent to \cite[Eq. 1.107]{Stanley2012}) gives the size of the conjugacy class corresponding to a tuple of invariant factors: 
\begin{equation}
  \label{eq:nqnn}
N_q(n,n;\I)=\frac{|\GL_n(\Fq)|}{\prod_{f\in S(\I)}c_f(\Phi_\I(f))}
\end{equation}
for any $n$-tuple $\I=(p_1,\ldots,p_n)$ of invariant factors with $\deg(p_1\cdots p_n)=n$. In the next section, we recall some basic facts about linear transformations defined on a subspace and then prove an extension (Theorem~\ref{snfcount}) of the above formula.

\section{Maps defined on a subspace with prescribed invariant factors}
\label{results}

Throughout this paper, we denote by $V$ an $n$-dimensional vector space over $\Fq$ and by $W$ a fixed $k$-dimensional subspace of $V$. We denote by $L(W,V)$ the space of all linear transformations from $W$ to $V$. Let $\Bk=\{v_1,\ldots,v_k\}$ be a fixed ordered basis for $W$ and let $\Bn=\{v_1,\ldots,v_n\}$ be an extension of this basis to an ordered basis of $V$. We define the invariant factors of a linear transformation defined on a subspace in terms of its matrix with respect to the bases $\Bk$ and $\Bn$.
\begin{definition}
Suppose $T\in L(W,V)$ and let $B\in M_{n,k}(\Fq)$ denote the matrix of $T$ with respect to the bases $\Bk$ and $\Bn$. The invariant factors of $T$ are the invariant factors of $xI_{n,k}-B$ viewed as a polynomial matrix in $M_{n,k}(\Fq[x])$.  
\end{definition}
Note that the above definition coincides with the usual definition of invariant factors of a linear transformation in the case $W=V$. The invariant factors of $T\in L(W,V)$ are independent of the specific choice of bases for $W$ and $V$ so long as the basis for $V$ is obtained as an extension of the basis for $W$. 
If $p_1,\ldots,p_k$ are the invariant factors of $T$, we write $\I_T=(p_1,\ldots,p_k)$. Though the terms `invariant factors' and `invariant polynomials' are often used interchangeably, some authors refer to the nonunit invariant factors of $T$ as the invariant polynomials of $T$.
\begin{remark}
While it is possible for an arbitrary element of $M_{n,k}(\Fq[x])$ to have less than $k$ invariant factors, a matrix polynomial of the form $xI_{n,k}-B$ for $B\in M_{n,k}(\Fq)$ always has $k$ invariant factors as $\delta_k(xI_{n,k}-B)\neq 0$. In fact, it is easily seen that $\delta_k(xI_{n,k}-B)$ is a polynomial of degree at most $k$.   
\end{remark}

\begin{definition}
For any $k$-tuple $\I$ of invariant factors, we define
$$
N_q(n,k;\I):=\#\left\{T\in L(W,V): \I_T=\I\right\}.
$$
\end{definition}

The above definition of $N_q(n,k;\I)$ is equivalent to the one given in the introduction. Given $T\in L(W,V)$, define
 $$\mathfrak{I}(T):= \mbox{ the maximal } T\mbox{-invariant subspace contained in } W.$$

The following proposition \cite[Thm. III.1.2]{Gohbergetal1995} relates the invariant factors of $T$ to the invariant factors of its restriction to $\mathfrak{I}(T)$.
\begin{proposition}
\label{prop:invpoly}
For $T\in L(W,V)$, let $T'$ denote the restriction of $T$ to its maximal invariant subspace $\IT$. Then $T$ and $T'$ have the same nonunit invariant factors.
\end{proposition}

\begin{corollary}
\label{cor:dim}
  For $T\in L(W,V)$ with $\I_T=(p_1,\ldots,p_k)$, we have
$$
\dim \mathfrak{I}(T)=\deg \prod_{i=1}^{k}p_i.
$$
\end{corollary}
\begin{proof}
Denote by $T'$ the restriction of $T$ to $\mathfrak{I}(T)$ and suppose $\dim \mathfrak{I}(T)=d$. Let $A\in M_{n,k}(\Fq)$ denote the matrix of $T$ with respect to the bases $\Bk$ and $\Bn$ and let $A'\in M_d(\Fq)$ denote the matrix of $T'$ with respect to some basis of $\IT$. The dimension of $\mathfrak{I}(T)$ is equal to the degree of the characteristic polynomial of $T'$, i.e.,
  \begin{align*}
    \dim \mathfrak{I}(T)=d&=\deg \det(xI_d-A')\\
                               &=\deg \delta_d(xI_d-A')\\
                               &=\deg \delta_k\left({xI_{n,k}-A}\right).
  \end{align*}
The last equality follows from Proposition \ref{prop:invpoly} since $T$ and $T'$ have the same nonunit invariant factors. Since $$\delta_k(xI_{n,k}-A)=\prod_{i=1}^{k}p_i,$$ the corollary follows.
\end{proof}

For subspaces $U\subseteq W \subseteq V$ and a tuple $\I$ of invariant factors, we define
$$
N(V,W,U;\I):=\#\{T\in L(W,V):\mathfrak{I}(T)=U\mbox{ and } \I_T=\I\}.
$$
Then \eqref{eq:nqnn} is equivalent to
\begin{equation}
\label{eq:nvvv}
N(V,V,V;\I)=\frac{|\GL_n(\Fq)|}{\prod_{f\in S(\I)}c_f(\Phi_\I(f))}  
\end{equation}
for any $n$-tuple $\I=(p_1\ldots,p_n)$ of invariant factors with $\deg \prod_{i=1}^np_i=n$. Note that $N(V,V,U;\I)=0$ for all $\I$ whenever $U$ is a proper subspace of $V$.

The next theorem gives an explicit formula for $N(V,W,U;\I)$ and is our main result. A crucial ingredient in the proof is the formula for the number of zero kernel pairs over a finite field. 

\begin{theorem}
\label{th:main}
Suppose $U$ is a subspace of $W$ and let $d, k, n$ denote the dimensions of $U,W,V$ respectively. Given a $k$-tuple $\I=(p_1,\ldots,p_k)$ of invariant factors with $\deg \prod_{i=1}^{k}p_i=d$, we have
\begin{equation*}
N(V,W,U;\I)=
\nqdd\prod_{i=d+1}^{k}(q^n-q^i).
\end{equation*}
\end{theorem}
\begin{proof}
 The case $W=V$ follows easily from \eqref{eq:nvvv}. Now suppose $W\subsetneq V$. Let $T\in L(W,V)$ with $\I_T=(p_1,\ldots,p_k)$. We need to count the number of such $T$ for which $\mathfrak{I}(T)=U$. Let $U'$ be a complementary subspace of $W$ with $W=U\oplus U'$. Let $\{v_1,\ldots,v_{d}\}$ be an ordered basis for $U$ and let $\{v_{d+1},\ldots,v_k\}$ be an ordered basis for $U'$ so that $\mathcal{B}=\{v_1,\ldots,v_k\}$ is a basis of $W$. We extend this basis further to a basis $\mathcal{B}'=\{v_1,\ldots,v_n\}$ of $V$. We now count linear transformations by counting their matrices with respect to the bases $\mathcal{B}$ and $\mathcal{B}'$. Accordingly, let ${A \brack C}\in M_{n,k}(\Fq)$ be the matrix of $T$ with respect to the bases $\mathcal{B}$ and $\mathcal{B}'$ for some $A\in M_k(\Fq)$ and $C\in M_{n-k,k}(\Fq)$. We partition $A$ and $C$ further into blocks corresponding to the dimensions of $U$ and $U'$ and write the matrix of $T$ in the form
$$
{A \brack C}=
\begin{bmatrix}
  A_{11}& A_{12}\\
A_{21}&   A_{22}\\
C_{11} & C_{12}
\end{bmatrix},
$$
where $A_{11}\in M_d(\Fq)$. Since $U$ is $T$-invariant, it follows that $A_{21}={\bf 0}\in M_{k-d,d}(\Fq)$ and $C_{11}={\bf 0}\in M_{n-k,d}(\Fq)$. 
We thus have
\begin{equation}
\label{partitioned}
{xI_k-A \brack -C}=
\begin{bmatrix}
xI_d- A_{11}& -A_{12}\\
\bf{0}&  xI_{k-d}- A_{22}\\
  \bf{0} & -C_{12}
\end{bmatrix}.
\end{equation}
We now need to count the number of tuples $(A_{11},A_{12},A_{22},C_{12})$ for which the invariant factors of \eqref{partitioned} are precisely $p_1,\ldots, p_k$. The fact that $\deg(p_1\cdots p_k)=d$ forces $p_i=1$ for $ i \leq k-d$. Since the matrix of $T$ restricted to $U$ is precisely $A_{11}$, it follows from Proposition \ref{prop:invpoly} that the invariant factors of $xI_d-A_{11}$ are precisely $p_{k-d+1},\ldots,p_k$, i.e., 
\begin{equation}
  \label{eq:a11}
  xI_d-A_{11}\sim \diag_{d,d}(p_{k-d+1},\ldots,p_k).
\end{equation}
Note that $\det(xI_d-A_{11})=\prod_{i=1}^{k}p_i$. Consider the $k\times k$ minors of the matrix on the right hand side of \eqref{partitioned}. For such a minor to be nonzero, it necessarily contains all rows of the submatrix $[ xI_d-A_{11} \quad -A_{12}]$. Thus, the nonzero $k\times k$ minors of ${xI_k-A \brack -C}$ are all divisible by $\det(xI_d-A_{11})$. Now observe that
 \begin{equation}
   \label{eq:deltas}
\delta_k\left({xI_k-A \brack -C}\right)=\det(xI_d-A_{11}) \cdot \delta_{k-d}\left({xI_{k-d}-A_{22} \brack -C_{12}}\right).   
 \end{equation}
Since $\det(xI_d-A_{11})=\prod_{i=1}^{k}p_i$ which also equals the left hand side of \eqref{eq:deltas}, it follows that
\begin{equation} 
\label{a22c12}
\delta_{k-d}\left({xI_{k-d}-A_{22} \brack -C_{12}}\right)=1.  
\end{equation}

From the preceding discussion, it is clear that if $A_{11}$ is chosen to satisfy \eqref{eq:a11}, $A_{22},C_{12}$ are chosen to satisfy \eqref{a22c12} and $A_{12}$ is chosen arbitrarily in $M_{d,k-d}(\Fq)$, then $T$ has the desired invariant factors. Conversely, for any $T$ with invariant factors $p_1,\ldots,p_k$, the matrices $A_{11},A_{12},A_{22},C_{12}$  necessarily satisfy \eqref{eq:a11} and \eqref{a22c12}. The number of choices for $A_{11}$ is precisely $N_q(d,d;\I_{(d)})$ where $\I_{(d)}$ denotes the $d$-tuple formed by taking the last $d$ coordinates of $\I$. The pair $(C_{12},A_{22})\in M_{n-k,k-d}(\Fq)\times M_{k-d}(\Fq)$ may be chosen to be any zero kernel pair. This can be done in $\prod_{i=1}^{k-d}(q^{n-d}-q^i)$ ways (\cite[Cor. 3.9]{zerokernel}) and $A_{12}$ can be chosen in $q^{d(k-d)}$ ways. Thus, the tuple $(A_{11},A_{12},A_{22},C_{12})$ can be chosen in 
\begin{equation}
\label{fixUandI}
q^{d(k-d)} N_q(d,d;\I_{(d)}) \prod_{i=1}^{k-d}(q^{n-d}-q^i)=  N_q(d,d;\I_{(d)}) \prod_{i=d+1}^{k}(q^{n}-q^i) 
\end{equation}
ways. It is easy to see that $S(\I)=S(\I_{(d)})$ and $\Phi_\I(f)=\Phi_{\I_{(d)}}(f)$ for all $f\in S(\I)$. The result now follows by substituting for $N_q(d,d;\I_{(d)})$ from \eqref{eq:nqnn}. 
\end{proof} 

%

\begin{remark}
If we set $U$ to be the zero subspace in the hypothesis of Theorem~\ref{th:main}, then $d=0$ and $\I$ is necessarily the $k$-tuple $(1,\ldots,1)$. In this case $\I_{(0)}$ denotes the null tuple and we adopt the convention that $N_q(0,0,\I_{(0)})=1$ and $|\GL_0(\Fq)|=1$ in the counting formula. 
 We then obtain
$$
N(V,W,\{0\},(1,\ldots,1))=\prod_{i=1}^{k}(q^n-q^i),
$$
which is an expression for the number of $T\in L(W,V)$ which have no nonzero invariant subspace (see \cite[Prob. 1.4]{zerokernel}).
\end{remark}

The number of $d$-dimensional subspaces of a $k$-dimensional vector space over $\Fq$ is well-known and given by the $q$-binomial coefficient
$$
{k \brack d}_q:=\frac{(q^k-1)(q^{k-1}-1)\cdots (q^{k-d+1}-1)}{(q^d-1)(q^{d-1}-1)\cdots (q-1)}.
$$
The following theorem gives a formula for the number of linear transformations defined on a subspace with a prescribed list of invariant factors.
\begin{theorem}
\label{snfcount}
Let $\I=(p_1,\ldots,p_k)$ be a $k$-tuple of invariant factors and suppose $\deg \prod_{i=1}^{k}p_i=d$. We have
  \begin{equation*}
    N_q(n,k;\I)={k \brack d}_q \nqdd \prod_{i=d+1}^{k}(q^n-q^i).
  \end{equation*}
\end{theorem}
\begin{proof}
We sum over all $d$-dimensional subspaces $U$ of $W$ to obtain
\begin{align*}
  N_q(n,k;\I)&=\sum_{\dim U=d}N(V,W,U;\I)\\
             &={k \brack d}_qN(V,W,U_0;\I),
\end{align*}
where $U_0$ is some fixed subspace of $W$ of dimension $d$. We now use Theorem~\ref{th:main} to deduce the result.
\end{proof}
The following result gives an expression for the number of linear transformations which have a given invariant subspace and will be used later on in the proof of Theorem \ref{Helmke}.
\begin{theorem}
If $U$ is a $d$-dimensional subspace of $W$, then
\label{givenU}
$$
\#\left\{T\in L(W,V): \mathfrak{I}(T)=U\right\}   = q^{d^2}\prod_{i=d+1}^{k}(q^n-q^i).
$$
\end{theorem}
\begin{proof}
  Summing \eqref{fixUandI} over all possible $k$-tuples $\I=(p_1,\ldots,p_k)$ of invariant factors with $\deg \prod_{i=1}^{k}p_i=d$, we obtain 
  \begin{align*}
\#\left\{T\in L(W,V): \mathfrak{I}(T)=U\right\}&
=\sum_{\I}N(V,W,U;\I)\\
&=\sum_{\I}N_q(d,d;\I_{(d)}) \prod_{i=d+1}^{k}(q^n-q^i),
  \end{align*}
where $\I_{(d)}$ denotes the $d$-tuple formed by the last $d$ coordinates of $\I$. This in turn equals
  \begin{equation*}
         \prod_{i=d+1}^{k}(q^n-q^i)   \sum_{\I}N_q(d,d;\I_{(d)})=q^{d^2}\prod_{i=d+1}^{k}(q^n-q^i).
  \end{equation*} 
The last equality follows from the fact that $\I_{(d)}$ varies over invariant factors corresponding to all possible conjugacy classes in $M_d(\Fq)$ and hence the above sum counts all elements of $M_d(\Fq)$.
\end{proof}
\begin{remark}
  \label{subspace_conjugacy}
In contrast with the case of linear operators, the invariant factors do not form a complete set of similarity invariants for a linear transformation defined only on a subspace. Recall that if $T\in L(W,V)$ and $T'\in L(W',V')$ denote two linear transformations defined on subspaces $W,W'$ of $V,V'$ respectively, then $T$ and $T'$ are similar if there exists a linear isomorphism $S:V\to V'$ with $S(W)=W'$ such that the following diagram commutes: 
\[
\begin{tikzcd}
W \arrow{r}{T} \arrow[swap]{d}{S} & V \arrow{d}{S}[swap]{\simeq} \\
W' \arrow{r}{T'} & V' 
\end{tikzcd}
\]
If $T$ and $T'$ are similar, then they necessarily have the same invariant factors but the converse is not true in general. We refer to Gohberg et al. \cite[Sec. III.3]{Gohbergetal1995} or Ferrer and Puerta \cite{JfFp1992} for a characterization of similarity invariants in this case.
\end{remark}

\section{Reachability subspace of a matrix pair}
\label{pairs}
In this section we assume that $k,n$ are positive integers with $k<n$. Let $\vx_0$ denote a fixed column vector in $\Fq^k$ and consider the sequence $\vx_s$ in $\Fq^k$ defined by
\begin{equation*}
  {\bf x}_{s+1}=A{\bf x}_s+B{\bf u}_s  \quad (s \geq 0),
\end{equation*}
where $A\in M_k(\Fq)$, $B\in M_{k,n-k}(\Fq)$ are fixed and $\vu_s(s\geq 0)$ is some sequence in $\Fq^{n-k}$. This sequence arises frequently in diverse areas such as control systems, linear sequential machines, discrete automata theory and convolutional error correcting codes just to name a few (see \cite{Forney1970}, \cite{Harrison1969}, \cite{Sontag1998}). The recurrence is readily solved for $\vx_k$ by backward substitution and we obtain
\begin{equation*}
  \vx_k=A^k\vx_0+
\begin{bmatrix}
B &AB &\cdots & A^{k-1}B
  \end{bmatrix}
\begin{bmatrix}
\vu_{k-1} \\
\vu_{k-2}\\
\vdots \\
\vu_0
  \end{bmatrix}.
\end{equation*}
It is often required that $\vx_k$ take all possible values in $\Fq^k$ by suitably varying the `input sequence' $\vu_s(0\leq s \leq k-1)$. This happens if and only if the \emph{reachability matrix} defined by
$$\mathcal{C}(A,B):=\begin{bmatrix}
B &AB &\cdots & A^{k-1}B
  \end{bmatrix}
$$
has full rank $k$. This motivates the definition of a reachable pair of matrices.
\begin{definition}
\label{reachable}
  The pair $(A,B)\in M_k(\Fq)\times M_{k,n-k}(\Fq)$ is said to reachable if
$$
\rank \mathcal{C}(A,B)=k.
$$
\end{definition}
The problem of counting the number of reachable pairs $(A,B)$ over a finite field is equivalent to counting the number of zero kernel pairs and is an interesting problem in its own right with connections to other combinatorial objects. We refer to \cite{zerokernel} for more on this topic. 
When the pair $(A,B)$ is not reachable, $\vx_k$ takes values in an affine subspace of dimension $r\leq k$ determined by $\vx_0$ and the reachability matrix $\mathcal{C}(A,B)$. We have the following natural generalization of the notion of reachability. 
\begin{definition}
\label{matrixpair}
 The pair $(A,B)\in M_k(\Fq)\times M_{k,n-k}(\Fq)$ is said to have an $r$-dimensional reachability subspace if
$$
\rank \mathcal{C}(A,B)=r.
$$
\end{definition}
A natural question at this point is to ask whether there is a simple formula for the number of matrix pairs with reachability subspace of a given dimension. This is the question addressed by Theorem \ref{Helmke}. We first establish the connection between linear transformations defined on a subspace and matrix pairs with reachability subspace of a given dimension.
\begin{proposition}
 Suppose $T\in L(W,V)$ and let the matrix of $T$ with respect to the bases $\Bk$ and $\Bn$ be ${A \brack C}\in M_{n,k}(\Fq)$ for some $A\in M_k(\Fq)$ and $C\in M_{n-k,k}(\Fq)$. We have an isomorphism of $\Fq$-vector spaces
$$
\mathfrak{I}(T)\simeq\bigcap_{i=0}^{k-1}\ker(CA^i).
$$
\end{proposition}
\begin{proof}
The proof is similar to that of \cite[Prop. 3.2]{zerokernel}. For any $v\in W$, let $\vx_v\in \Fq^k$ denote the coordinate column vector of $v$ with respect to the basis $\Bk$. If $v\in \IT$, then $T^iv$ lies in $\IT$ and hence in $W$ for all $i$. The coordinate vector of $Tv$ with respect to $\Bn$ is
$$
{A \brack C}\vx_v={A\vx_v \brack C\vx_v}.
$$
Since $Tv\in W$, it follows that $C\vx_v=0$ and the coordinate vector of $Tv$ with respect to $\Bk$ is simply $A\vx_v$. By considering $T^2v$ it follows similarly that $CA\vx_v=0$ and hence the coordinate vector of $T^2v$ with respect to $\Bk$ is $A^2\vx_v$. Continuing this line of reasoning, we find that $CA^i\vx_v=0$ for $0\leq i \leq k-1$. Thus
$$
\vx_v\in \bigcap_{i=0}^{k-1}\ker(CA^i),
$$
which proves that
\begin{equation}
\label{incl}
\{\vx_v:v\in\mathfrak{I}(T)\}\subseteq\bigcap_{i=0}^{k-1}\ker(CA^i).  
\end{equation}
Conversely, suppose ${\bf x}\in \bigcap_{i=0}^{k-1}\ker(CA^i)$. 
Let $\chi_A(x)$ denote the characteristic polynomial of $A$. If $u$ is the vector whose coordinates with respect to $\Bk$ are precisely ${\bf x}$, then it is easily verified that $\chi_A(T)u$ is well-defined and equal to 0. It follows that the $\Fq$-linear subspace spanned by $u, Tu, \ldots, T^{k-1}u$ is invariant under $T$. Therefore $u$ generates a $T$-invariant subspace and hence $u\in \IT$. Thus the reverse inclusion in \eqref{incl} holds as well. We thus have
$$
\{\vx_v:v\in\mathfrak{I}(T)\}=\bigcap_{i=0}^{k-1}\ker(CA^i),
$$
from which the proposition immediately follows.
\end{proof}
\begin{corollary}
\label{dimit}
Suppose ${A \brack C}\in M_{n,k}(\Fq)$ denotes the matrix of $T\in L(W,V)$. We have 
\begin{align*}
\dim \IT&=\dim \bigcap_{i=0}^{k-1}\ker(CA^i)  \\
&=k-\rank\begin{bmatrix}
      C \\
      CA \\
      \vdots \\
      CA^{k-1}
    \end{bmatrix}.
\end{align*}
\end{corollary}
\begin{proof}
  The second equality follows from the rank-nullity theorem since the null space of the block matrix above is precisely $\cap_{i=0}^{k-1}\ker(CA^i)$.
\end{proof}

\begin{corollary}
\label{link}
The pair $(A,B)\in M_k(\Fq)\times M_{k,n-k}(\Fq)$ has a reachability subspace of dimension $k-\dim \IT$ where $T:W\to V$ is the linear transformation defined by declaring the matrix of $T$ with respect to $\Bk$ and $\Bn$ to be ${A^T \brack B^T}$.
\end{corollary}
\begin{proof}
  Follows from Definition \ref{matrixpair} and Corollary \ref{dimit}.
\end{proof}
We are now ready to prove the result of Helmke et al. \cite[Thm. 2]{Helmkeetal2015} on matrix pairs with reachability subspace of a given dimension.
\begin{theorem}
\label{Helmke}
The number of pairs $(A,B)\in M_k(\Fq)\times M_{k,n-k}(\Fq)$ with $r$-dimensional reachability subspace is given by
$$
{k \brack r}_q q^{(k-r)^2}\prod_{i=k-r+1}^{k}(q^n-q^i).
$$
\end{theorem}
\begin{proof}
  The number of pairs $(A,B)$ with $r$-dimensional reachability subspace is equal to 
$$
\#\left\{(A,B)\in M_k(\Fq)\times M_{k,n-k}(\Fq):\rank \mathcal{C}(A,B)=r\right\}.\\
$$
By Corollary \ref{link}, this equals
$$
\#\left\{T\in L(W,V): \dim \mathfrak{I}(T)=d\right\},
$$
where $d=k-r$. As $\mathfrak{I}(T)$ could be any $d$-dimensional subspace of $W$, the number of possibilities for $\mathfrak{I}(T)$ is clearly ${k \brack d}_q={k \brack r}_q$. The result now follows from Theorem~\ref{givenU}.
\end{proof}

Setting $r=k$ in the above theorem, we find that the number of reachable pairs $(A,B)\in M_k(\Fq)\times M_{k,n-k}(\Fq)$ is given by
$$
\prod_{i=1}^k(q^n-q^i).
$$

\begin{remark}
Let $\I_1$ denote the $k$-tuple $(1, \ldots ,1)$. The formula for the number of reachable pairs is equivalent to
$$
N_q(n,k;\I_1)=\prod_{i=1}^{k}(q^n-q^i).
$$
\end{remark}


\section{An extension of the Gerstenhaber-Reiner formula}
\label{sec:ghformula}
It is a remarkable fact that the number of nilpotent matrices in $M_n(\Fq)$ is given by a rather simple formula, namely, $q^{n(n-1)}$. This result is often attributed to Fine and Herstein \cite{FH1958} but appears to have been known earlier to Philip Hall who gave two different proofs in a series of lectures at the Edinburgh mathematical colloquium at St. Andrews in 1955 (see \cite[p. 618]{PhilipHall1984}). Subsequently, the result was generalized independently by Gerstenhaber \cite{Mg1961} and Reiner \cite[Thm. 2]{Ir1961} who gave a formula for the number of matrices with a given characteristic polynomial. In this section, we extend the result of Gerstenhaber and Reiner to the setting of rectangular matrices.

Let $\Phi_n:M_n(\Fq)\to \Fq[x]$ denote the map which associates to a square $n\times n$ matrix its characteristic polynomial, i.e., $\Phi_n(A)=\det(xI_n-A)$. The fibers of $\Phi_n$ can be computed explicitly as the following theorem shows.
\begin{theorem}[Gerstenhaber-Reiner]
For a monic polynomial $f(x)\in \Fq[x]$ of degree $n$, let $f=f_1^{e_1}\cdots f_t^{e_t}$ denote the canonical factorization of $f$ into distinct irreducible polynomials $f_i$ of degree $d_i(1\leq i \leq t)$ in $\Fq[x]$. The number of matrices in $M_n(\Fq)$ with characteristic polynomial $f$ is given by
$$
|\Phi_n^{-1}(f)|=q^{n^2-n}\frac{F(q,n)}{\prod_{i=1}^{t}F(q^{d_i},e_i)},
$$
where $F(q,r)=\prod_{i=1}^{r}(1-q^{-i})$ for each nonnegative integer $r$.
\end{theorem}
We extend this result to rectangular matrices as follows. Define 
$$\Phi_{n,k}:M_{n,k}(\Fq)\to \Fq[x]$$
by
\begin{equation*}
  \Phi_{n,k}(A)=\delta_k\left(xI_{n,k}-A\right),
\end{equation*}
where $\delta_k(\cdot)$ denotes the $k$\textsuperscript{th} determinantal divisor as before. It is clear that $\Phi_{n,n}$ is just the map $\Phi_n$. The following result gives a formula for the size of the fibers of $\Phi_{n,k}$ and may be viewed as an extension of the Gerstenhaber-Reiner formula. 
 \begin{theorem}
\label{extension}
  For a monic polynomial $f(x)\in \Fq[x]$ of degree $d\leq k$, let $f=f_1^{e_1}\cdots f_t^{e_t}$ be the canonical factorization of $f$ into distinct irreducibles $f_i$ of degree $d_i(1\leq i \leq t)$. We have 
  \begin{equation*}
\label{eq:extension}
   \left| \Phi^{-1}_{n,k}(f)\right|= {k \brack d}_q\frac{q^{d^2-d}F(q,d)}{\prod_{i=1}^{t}F(q^{d_i},e_i)}\prod_{i=d+1}^{k}(q^n-q^i).
  \end{equation*}
 \end{theorem}
 \begin{proof}
We have $\Phi_{n,k}(A)=f$ precisely when the product of the diagonal elements appearing in the Smith form of $xI_{n,k}-A$ is $f$. Therefore, if we sum over all possible $k$-tuples $(p_1,\ldots,p_k)$ of monic polynomials in $\Fq[x]$ with $p_i\mid p_{i+1}(1\leq i \leq k-1)$ and $p_1\cdots p_k=f$, we obtain
   \begin{align*}
     \left| \Phi^{-1}_{n,k}(f)\right|&=\sum_{\substack{p_1\cdots p_k=f\\ p_i \mid p_{i+1}}}N_q\left(n,k;(p_1,\ldots,p_k)\right)\\
                &={k \brack d}_q \prod_{i=d+1}^{k}(q^n-q^i)\sum_{\substack{p_1\cdots p_k=f\\ p_i \mid p_{i+1}}}N_q(d,d;(p_{k-d+1},\ldots, p_k)),
   \end{align*}
  where the second equality follows from Theorem \ref{snfcount} and \eqref{eq:nqnn}. The sum in the last expression counts all $d\times d$ matrices for which the product of invariant factors is $f$ and, therefore, is clearly equal to the number of matrices in $M_d(\Fq)$ with characteristic polynomial $f$. Thus
\begin{align*}
   \left| \Phi^{-1}_{n,k}(f)\right|&={k \brack d}_q\left|\Phi_d^{-1}(f)\right| \prod_{i=d+1}^{k}(q^n-q^i),
\end{align*}
and the theorem follows by applying the Gerstenhaber-Reiner formula to $|\Phi_d^{-1}(f)|$. 
\end{proof}

We conclude this section with an interesting application of Theorem \ref{extension}. Consider the following combinatorial problem: Given an $n$-dimensional vector space $V$ over $\Fq$ and a $k$-dimensional subspace $W$ of $V$, how many linear transformations $T:W\to V$ have the property that $T$ can be extended to a nilpotent transformation on all of $V$? We shall show that the answer is $q^{n(k-1)}(q^n-q^k+1)$. Note that this extends Hall's result 
which says that the answer  for $k=n$ must be $q^{n(n-1)}$. Before we proceed, we need a lemma.

  \begin{lemma}
\label{identity}
For any nonnegative integer $d$ and indeterminate $y$,
$$
y^d=\sum_{j=0}^{d}q^{j^2}\prod_{i=j+1}^{d}(y-q^i){d \brack j}_q.
$$
\end{lemma}
\begin{proof}
  We begin with the identity (see \cite[Ch. 4, Eq. 4.4]{CK2002})  
  \begin{align*}
z^d
   &=\sum_{j=0}^{d}{d \brack j}_q\prod_{l=0}^{j-1}(z-q^l)\\
   &=\sum_{j=0}^{d}q^{j(d-j)}{d \brack j}_{1/q}\prod_{l=0}^{j-1}(z-q^l)\\
   &=\sum_{j=0}^{d}q^{j(d-j)}{d \brack j}_{1/q}\prod_{i=d-j+1}^{d}(z-q^{d-i}).\\
  \end{align*}
 Since the left hand side is independent of $q$, we replace $q$ by $1/q$ to obtain
\begin{align*}
  z^d&=\sum_{j=0}^{d}q^{j(j-d)}{d \brack j}_q\prod_{i=d-j+1}^{d}(z-q^{i-d})\\
     &=\sum_{j=0}^{d}q^{j(j-2d)}{d \brack j}_q\prod_{i=d-j+1}^{d}(q^dz-q^i).\\
\end{align*}
Put $q^dz=y$ to obtain
  \begin{align*}
    y^d&=q^{d^2}\sum_{j=0}^{d}q^{j(j-2d)}{d \brack j}_q\prod_{i=d-j+1}^{d}(y-q^i)\\
       &=\sum_{j=0}^{d}q^{(d-j)^2}{d \brack j}_q\prod_{i=d-j+1}^{d}(y-q^i)\\
       &=\sum_{j=0}^{d}q^{j^2}{d \brack j}_q\prod_{i=j+1}^{d}(y-q^i)
  \end{align*}
as desired.
\end{proof}

\begin{theorem}
Let $V$ be an $n$-dimensional vector space over $\Fq$ and let $W$ be a $k$-dimensional subspace of $V$. The number of linear transformations $T:W\to V$ that can be extended to a nilpotent transformation on all of $V$ is given by
$$q^{n(k-1)}(q^n-q^k+1).$$
\end{theorem}    
\begin{proof}
Let $\Delta_q(k,n)$ denote the number of linear transformations $T:W\to V$ which can be extended to a nilpotent transformation on $V$. It follows from Wimmer's theorem \cite[Th. 15]{Cravo2009} that the product of invariant factors of any such $T$ is a divisor of $x^n$ and is of degree at most $k$. Therefore,
\begin{align*}
 \Delta_q(k,n)&=\sum_{l=0}^{k}|\Phi^{-1}_{n,k}(x^l)| \\
              &=\sum_{l=0}^{k}|\Phi^{-1}(x^l)|{k \brack l}_q \prod_{i=l+1}^{k}(q^n-q^i)\\
          &=\sum_{l=0}^{k}q^{l(l-1)}{k \brack l}_q \prod_{i=l+1}^{k}(q^n-q^i),
\end{align*}
where we have used Theorem \ref{extension} for the second equality. We set $q^n=y$ and use Lemma \ref{identity} to see that the last expression is equivalent to
\begin{align*}
\Delta_q(k,n)&=y^k-\sum_{l=0}^{k}(q^{l^2}-q^{l(l-1)}){k \brack l}_q \prod_{i=l+1}^{k}(y-q^i)\\
             &=y^k-\sum_{l=1}^{k}q^{l(l-1)}(q^k-1){k-1 \brack l-1}_q \prod_{i=l+1}^{k}(y-q^i)\\
         &=y^k-\sum_{m=0}^{k-1}q^{m(m+1)}(q^k-1){k-1 \brack m}_q \prod_{i=m+2}^{k}(y-q^i)\\
         &=y^k-(q^k-1)\sum_{m=0}^{k-1}q^{m(m+1)}{k-1 \brack m}_q q^{k-m-1}\prod_{i=m+2}^{k}(y/q-q^{i-1})\\
        &=y^k-(q^k-1)q^{k-1}\sum_{m=0}^{k-1}q^{m^2}{k-1 \brack m}_q \prod_{j=m+1}^{k}(y/q-q^j)\\
&=y^k-(q^k-1)q^{k-1}(y/q)^{k-1}\\
&=y^k-(q^k-1)y^{k-1},
\end{align*}
where the penultimate equality again follows from Lemma \ref{identity}. We now resubstitute $y$ by $q^n$ to obtain the result.
\end{proof}



\section*{Acknowledgements}
This work was partly done while the author was a postdoctoral fellow at the Harish-Chandra Research Institute, Allahabad and partly at the Indraprastha Institute of Information Technology, Delhi.

\bibliographystyle{plain}
\bibliography{autobib}

\end{document}